\documentclass[a4paper]{amsart}
\usepackage{amssymb, amsmath, amscd, mathrsfs,marvosym, psfrag,color}

\input xy
\xyoption{all}
\DeclareMathAlphabet{\mathpzc}{OT1}{pzc}{m}{it}

\newtheorem{theorem}{Theorem}[section]

\newtheorem{proposition}[theorem]{Proposition}

\newtheorem{lemma}[theorem]{Lemma}

\newtheorem*{Th}{Theorem}

\theoremstyle{definition}
\newtheorem{definition}[theorem]{Definition}

\theoremstyle{remark}

\newcommand{\CA}{{\mathcal A}}

\newcommand{\CN}{{\mathcal N}}
\newcommand{\CO}{{\mathcal O}}

\newcommand{\CV}{{\mathcal V}}

\newcommand{\fg}{{{\mathfrak g}}}

\newcommand{\fsl}{{{\mathfrak sl}}}

\newcommand{\Fr}{\operatorname{Fr}}

\newcommand{\tK}{{\widetilde{K}}}
\newcommand{\tbV}{{\widetilde{\bV}}}

\newcommand{\tJ}{{\widetilde{ J }}}
\newcommand{\tI}{{\widetilde{ I}}}

\newcommand{\DC}{{\mathbb C}}

\newcommand{\DZ}{{\mathbb Z}}

\newcommand{\bD}{{\mathbf D}}

\newcommand{\bU}{{\mathbf U}}

\newcommand{\bV}{{\mathbf V}}
\newcommand{\bT}{{\mathbf T}}

\newcommand{\bZ}{{\mathbf Z}}

\newcommand{\ch}{{\operatorname{ch}\, }}

\newcommand{\End}{{\operatorname{End}}}
\newcommand{\Ext}{{\operatorname{Ext}}}
\newcommand{\Hom}{{\operatorname{Hom}}}

\newcommand{\Res}{{\operatorname{Res}\,}}

\newcommand{\catmod}{{\operatorname{-mod}}}

\newcommand{\im}{{\operatorname{im}\,}}

\newcommand{\ev}{\operatorname{ev}}
\newcommand{\coev}{\operatorname{coev}}
\newcommand{\id}{{\operatorname{id}}}

\newcommand{\lgl}{\langle}
\newcommand{\rgl}{\rangle}

\newcommand{\comment}[1]{}

\begin{document}

\title[]{Periodicity for subquotients of the modular category $\CO$} \author[]{Peter Fiebig}
\begin{abstract}  In this paper we study the  category $\CO$ over the hyperalgebra of a reductive algebraic group in positive characteristics. For any locally closed subset $K$ of weights we define a  subquotient $\CO_{[K]}$ of $\CO$. It has the property that its simple objects are parametrized by elements in $K$. We then show that $\CO_{[K]}$ is equivalent to $\CO_{[K+p^l\gamma]}$ for any dominant weight $\gamma$ if $l>0$ is an integer  such that $K\cap (K-p^l\eta)=\emptyset$ for all weights $\eta> 0$. Hence it is enough to understand the subquotients inside the dominant (or the antidominant) chamber.
\end{abstract}

\address{Department Mathematik, FAU Erlangen--N\"urnberg, Cauerstra\ss e 11, 91058 Erlangen}
\email{fiebig@math.fau.de}
\maketitle

\section{Introduction}

One of the cornerstones of the rational  representation theory of a reductive algebraic group $G$ over a field $k$ of characteristic $p>0$  is Steinberg's tensor product theorem. For a dominant weight $\mu$ we denote by $L(\mu)$ the irreducible representation of $G$ with highest weight $\mu$. For a dominant weight $\lambda$ there is a $p$-adic extension $\lambda=\lambda_0+p\lambda_1+p^2\lambda_2+\dots+p^l\lambda_l$ with uniquely defined restricted weights $\lambda_0,\dots,\lambda_l$, and  Steinberg's tensor product theorem states that  
$$
L(\lambda)\cong L(\lambda_0)\otimes L(\lambda_1)^{[1]}\otimes\dots\otimes L(\lambda_l)^{[l]},
$$
where $(\cdot)^{[n]}$ denotes the $n$-fold Frobenius twist. 
In order to prove the theorem, it is, by induction, sufficient to prove the following. If $\lambda_0$ is restricted and $\gamma$ is dominant, then $L(\lambda_0)\otimes L(\gamma)^{[l]}$ is a simple $G$-module (it must then be isomorphic to $L(\lambda_0+p^l\gamma)$ for weight reasons). The theorem has numerous generalizations (for non-dominant weights, quantum groups, or for representations of the hyperalgebra of $G$, see, for example, \cite{A}).

The above serves as a motivation for us to consider the functor $(\cdot)\otimes L(\gamma)^{[l]}$ on the category of modules of the hyperalgebra (or algebra of distributions) $U$ associated with $G$. This algebra coincides with the universal enveloping algebra of the Lie algebra of $G$ in the case that the ground field is of characteristic $0$, but this is not the case in positive characteristics. The category of rational representations of $G$ can be identified with the category of finite dimensional $U$-modules. As $U$ admits a triangular decomposition one can, as in the classical characteristic $0$ case, drop the finite dimensionality condition and instead consider the {\em highest weight category $\CO$}\footnote{for simplicity we only consider in this article modules with integral weights.}.  The category $\CO$ for $U$ shares many familiar properties with its characteristic $0$ relative, which was studied intensely over multiple decades (for an overview on the most essential results, see \cite{HBook}). But in the modular case it has an additional property that does not occur in characteristic $0$, and which is closely connected to Steinberg's theorem: it inhibits a periodicity structure on subquotients. 

In this article we consider subquotient categories $\CO_{[K]}$ of $\CO$ associated to locally closed subsets of the weight lattice $X$ (a locally closed subset of $X$ is a union of intervals $[\lambda,\mu]$ with respect to the usual partial order).  We then prove the following.

\begin{Th} Suppose that $K$ is a locally closed quasi-bounded subset of $X$. Suppose that $l>0$ is such that $K\cap (K+p^l\eta)=\emptyset$ for all weights $\eta>0$. Then the functor $(\cdot)\otimes L(\gamma)^{[l]}$ induces an equivalence $\CO_{[K]}\xrightarrow{\sim}\CO_{[K+p^l\gamma]}$ for all dominant weights $\gamma\in X$.
\end{Th} 
Apart from Steinberg's tensor product theorem  also the periodicity of $(p,\Delta)$-filtrations (cf. Corollary 4.3 in \cite{A}) could find a conceptual foundation in the above theorem. 
  
Note that if $K$ is finite, then an integer $l$ as required in the theorem always exists. Once such an $l$ is fixed, we still can choose an arbitrary dominant weight $\gamma$ for the statement to hold. This shows that we can transfer the subquotient $\CO_{[K]}$ deep inside the dominant chamber, or, by reversing the statement, deep inside the antidominant chamber. In both cases, the representation theory has some valuable additional features. By \cite{A}, projective covers exist in $\CO$ for the simple objects $L(\lambda)$ for antidominant weights $\lambda$, while the tilting modules $T(\lambda)$  exist in $\CO$ for all dominant $\lambda$. 

We use this opportunity to also state and prove some basic results on the modular category $\CO$ for future reference, such as Krull-Remak-Schmidt decompositions, the existence of projectives in truncated subcategories, BGGH-reciprocity, and some additional results on modules admitting a Verma flag. Most of these results are well-known in the characteristic $0$ case and the proofs can be found in \cite{HBook}. If the proof in the modular case does not need adjustment, we will simply refer to the appropriate result in \cite{HBook}. We give a more detailed proof if some adjustments are required.

\section{The modular category $\CO$}
Let $\fg$ be a semisimple complex Lie algebra  with root system $R$. For $\alpha\in R$ we denote by $\alpha^\vee$ the associated coroot, and by $X$ the weight lattice.  We fix a basis $\Pi\subset R$ and  denote by $R^+\subset R$ the corresponding system of positive roots.  Then we denote by $\le$ the induced partial order on $X$, i.e. $\mu\le\lambda$ if and only if $\lambda-\mu$ can be written as a sum of elements of $R^+$. We let $U_\DC=U(\fg)$ be the universal enveloping algebra of $\fg$.

\subsection{Kostant's integral form} 
Let $\{e_\alpha,f_\alpha,h_\beta\mid\alpha\in R^+,\beta\in \Pi\}$ be a Chevalley basis of $\fg$. By $U_\DZ\subset U_\DC$ we denote Kostant's integral form (cf. \cite{K}). Recall that $U_\DZ$ is the unital subring of $U_\DC$ that is generated by the elements $e_\alpha^{(n)}:=e_\alpha^n/n!$, $f_\alpha^{(n)}:=f_\alpha^n/n!$ and ${{h_\beta}\choose n}:=\frac{h_\beta(h_{\beta}-1)\dots(h_{\beta}-n+1)}{n!}$ with $n> 0$, $\alpha\in R^+$ and $\beta\in\Pi$. We denote by $U_\DZ^+$ the subring generated by the $e_{\alpha}^{(n)}$'s, by $U_\DZ^-$ the subring generated by the  $f_{\alpha}^{(n)}$'s, and by $U_\DZ^0$ the subring generated by the ${{h_\beta}\choose n}$'s. The following integral version of the PBW-theorem is one of the main results in \cite{K}.

\begin{theorem} \label{thm-PBW}
\begin{enumerate}
\item The algebra $U_\DZ^+$ is free over $\DZ$, and the elements $\prod_{\alpha\in R^+}e_\alpha^{(n_\alpha)}$ with $n_\alpha\ge 0$ form a basis.
\item The algebra $U_\DZ^-$ is free over $\DZ$, and the elements $\prod_{\alpha\in R^+}f_\alpha^{(n_\alpha)}$ with $n_\alpha\ge 0$ form a basis.
\item The algebra $U_\DZ^0$ is free over  $\DZ$, and the elements $\prod_{\beta\in \Pi}{{h_\beta}\choose{n_\beta}}$  with $n_\beta\ge 0$  form a basis. 
\item The multiplication defines an isomorphism 
$$
U_\DZ^-\otimes_\DZ U_\DZ^0\otimes_\DZ U_\DZ^+\xrightarrow{\sim} U_\DZ
$$
of $\DZ$-modules. 
\end{enumerate}
\end{theorem}
The products in parts (1) and (2)  above should be taken with respect to a fixed, but arbitrary order on $R^+$. The algebra $U^0_\DZ$ is commutative, so no order is needed. 

Note that $U_\DZ$ can be considered as the algebra of distributions of the semisimple and simply connected  $\DZ$-group scheme $G_\DZ$ associated with $R$ (cf. Section II.1.12 in \cite{J}).  It inherits a Hopf algebra structure from $U_\DC$.

\subsection{The modular category $\CO$} 
From now on we fix a field $k$ of characteristic $p>0$. We set $U:=U_\DZ\otimes_\DZ k$ and define $U^+$, $U^-$, $U^0$ likewise. As before, we can consider $U$ as the algebra of distributions of the group scheme $G_k$. It inherits a Hopf algebra structure from $U(\fg)$ via $U_\DZ$. The category of $U$-modules hence obtains a tensor product structure $(\cdot)\otimes(\cdot)$. 

To any $\lambda\in X$ we can associate a character $\chi_\lambda\colon U^0\to k$ that maps ${h_\beta}\choose{n}$ to the image of ${{\lgl\lambda,\beta^\vee\rgl}\choose n}$ in $k$.
For a $U^0$-module $M$ and $\lambda\in X$ we define
$$
M_\lambda:=\{m\in M\mid h.m=\chi_\lambda(h) m\text{ for all $h\in U^0$}\}
$$
and call this the {\em $\lambda$-weight space} of $M$. We say that $M$ is a {\em weight module} if $M=\bigoplus_{\lambda\in X}M_\lambda$, and we say that $\lambda$ is a {\em weight of $M$} if $M_\lambda\ne0$. We say that  a $U^+$-module $M$ is {\em $U^+$-locally finite}  if $M$ is the union of its finite dimensional $U^+$-submodules. 

\begin{definition} 
The category $\CO$ is the full subcategory of the category of  $U$-modules that contains all objects that are weight modules and $U^+$-locally finite. 
\end{definition}
If $M$ is an object in $\CO$, then every submodule and every quotient of $M$ is also contained in $\CO$, and  $\CO$ is an abelian category. It is easy to check that $\CO$ is stable under the tensor product on $U$-modules. 

\subsection{Verma modules and simple quotients}
The algebra  $U^0$ normalizes $U^+$, so    $U^{\ge 0}:=U^0U^+\subset U$ is a subalgebra. 
For any $\lambda\in X$ we denote by $k_\lambda$ the one-dimensional $U^{\ge0}$-module on which $U^0$ acts via the character $\chi_\lambda$, and $U^{+}$ acts via the augmentation $U^{+}\to k$ (that sends all generators $e_{\alpha}^{(n)}$ to $0$ for $n>0$). The induced $U$-module
$$
\Delta(\lambda):=U\otimes_{U^{\ge 0}} k_\lambda
$$
is called the {\em Verma module with highest weight $\lambda$}. We denote by $v_\lambda=1\otimes 1$ the obvious generator of $\Delta(\lambda)$. The following proposition collects the basic facts about Verma modules. The arguments for the proofs are standard and do not depend on the characteristic of $k$. The characteristic $0$ version can be found in \cite{HBook}.  
\begin{proposition}\label{prop-basicVerma}
\begin{enumerate}
\item For each $\lambda\in X$ the $U$-module $\Delta(\lambda)$ is contained in $\CO$.
\item For any object $M$ in $\CO$, the map $\Hom_\CO(\Delta(\lambda),M)\to M_\lambda$, $f\mapsto f(v_\lambda)$ is injective with image $\{m\in M_\lambda\mid e^{(n)}_\alpha.m=0\text{ for all $\alpha\in R^+$, $n>0$}\}$. 
\item For each $\lambda\in X$ there exists a unique simple quotient $L(\lambda)$ of $\Delta(\lambda)$ in $\CO$.
\item The set $\{L(\lambda)\}_{\lambda\in X}$ is a full set of representatives  for the simple isomorphism classes of $\CO$. 
\end{enumerate}
\end{proposition}

\subsection{Finite dimensional simple modules}
Denote by $X^+=\{\lambda\in X\mid \lgl\lambda,\alpha^\vee\rgl\ge 0\text{ for all $\alpha\in \Pi$}\}$ the set of {\em dominant weights}. For a dominant weight $\lambda$ we denote by $V(\lambda)$ the {\em Weyl module} with highest weight $\lambda$. It is constructed as follows. We denote by $L_\DC(\lambda)$ the simple, finite dimensional $U_\DC$-module with highest weight $\lambda$ and choose a non-zero vector $v\in L_\DC(\lambda)$. The integral Weyl module is then defined by $V_\DZ(\lambda)=U_\DZ.v\subset L_\DC(\lambda)$. It turns out that this is a free abelian group of finite rank inside $L_\DC(\lambda)$ and $V_\DZ(\lambda)\otimes_\DZ\DC=L_\DC(\lambda)$. The Weyl module in the modular case is then $V(\lambda)=V_\DZ(\lambda)\otimes_\DZ k$. It is a finite dimensional $U$-module with highest weight $\lambda$ and it is contained in $\CO$.   
\begin{lemma} The object $L(\lambda)$ is finite dimensional if and only if $\lambda$ is a dominant weight.
\end{lemma}
\begin{proof} If $\lambda$ is a dominant weight, then the Weyl module $V(\lambda)$ is a finite dimensional  $U$-module with highest weight $\lambda$ and it contains $L(\lambda)$ as a quotient. Hence $L(\lambda)$ is finite dimensional. If $\lambda$ is not dominant, then there is a simple root $\alpha$ such that $\lgl\lambda,\alpha^\vee\rgl < 0$. We consider the $\alpha$-string in  $L(\lambda)$ through the highest weight: $\bigoplus_{n\ge 0}L(\lambda)_{\lambda-n\alpha}$. This is a highest weight module for the subalgebra $U^{\alpha}$ of $U$ generated by the $e_{\alpha}^{(n)}$, $f_{\alpha}^{(n)}$, ${h_{\alpha}\choose {n}}$ and $n\ge 0$. The algebra $U^{\alpha}$ is the algebra of distributions associated to the Lie algebra $\fsl_2(\DC)$ and the field $k$. The structure of the highest weight modules can be worked out explicitely (for example, see \cite{A}). The fact that $\lgl\lambda,\alpha^\vee\rgl < 0$ implies that $\bigoplus_{n\ge 0}L(\lambda)_{\lambda-n\alpha}$ must be infinite dimensional, and hence so is $L(\lambda)$. \end{proof}

\subsection{Characters} 
Suppose that $M$ is an object in $\CO$ that has the property that its set of weights  is bounded, i.e. there exists some $\gamma\in X$ such that  $M_\lambda\ne 0$ implies $\lambda\le\gamma$. Suppose furthermore that each weight space of $M$ is finite dimensional. Then one can define the character of $M$ as
$$
\ch  M=\sum_{\lambda\le\gamma}(\dim_k M_\lambda)e^{\lambda}\in\widehat{\DZ[X]},
$$
where $\widehat{\DZ[X]}$ is a suitable completion of the group ring $\DZ[X]$.  This applies to any finitely generated object in $\CO$ and, in particular, to the simple objects $L(\lambda)$, so there is a well-defined character $\ch  L(\lambda)$. Clearly, $\ch  L(\lambda)\in e^{\lambda}+\sum_{\mu<\lambda} \DZ_{\ge 0}  e^{\mu}$. Note that for any $\nu,\gamma\in X$ the interval $[\nu,\gamma]$ is finite. If $M$ is as above, there are well defined numbers $a_\mu\in \DZ_{\ge 0} $ such that
$$
\ch  M=\sum_{\mu\in X}a_\mu \ch  L(\mu).
$$
We write $(M: L(\mu))=a_\mu$ and call this the {\em multiplicity of $L(\mu)$ in $M$}.

\subsection{A duality on $\CO$}
Recall that there is an antiautomorphism $\sigma$ on $U_\DC$ that maps $e_{\alpha}$ to $f_\alpha$ and $h_\alpha$ to $h_\alpha$. It hence leaves $U_\DZ$ stable, so we obtain an antiautomorphism on $U$ that we denote by the same symbol. For a $U$-module $M$ we denote by $M^{\sigma}$ its twist by $\sigma$. Then $(M^\sigma)_\lambda=M_\lambda$ as a vector space, so the $\sigma$-twist preserves weight modules. For a $U$-module $M$ we denote by $M^\ast=\bigoplus_{\lambda\in X}\Hom_k(M_\lambda,k)$ the restricted linear dual. It acquires the structure of a $U$-module by setting $(x.\phi)(m)=\phi(\sigma(x).m)$ for all $m\in M$, $\phi\in M^\ast$ and $x\in U$. 

\begin{lemma} Suppose that  $M$ is an object in $\CO$ with finite dimensional weight spaces and such that  its set of weights is bounded from above .  Then $M^\ast$ is an object in $\CO$ as well and we have a functorial  identification $(M^{\ast})^\ast=M$. 
\end{lemma} 
\begin{proof} As observed above, $M^\ast$ is a weight module.  As the set of weights of $M^\ast$ coincides with the set of weights of $M$, it is bounded from above, hence $M^\ast$ must be $U^+$-locally finite. So it is contained in $\CO$ as well. It again satisfies the assumptions of the above lemma, and it is immediate that $(M^{\ast})^\ast=M$ canonically. 
\end{proof}
For the objects in the statement 
Clearly, $L^\ast$ is a simple object in $\CO$ if $L$ is. From the fact that  the characters agree we deduce that $L(\lambda)^\ast\cong L(\lambda)$ for all $\lambda\in X$. 

We denote by $\nabla(\lambda)=\Delta(\lambda)^\ast$ the dual of the Verma module with highest weight $\lambda$.
\begin{lemma} \label{lemma-ExtVan} We have $\dim_k\Hom_{\CO}(\Delta(\lambda),\nabla(\mu))=\delta_{\lambda,\mu}$ and $\Ext_\CO^1(\Delta(\lambda),\nabla(\mu))=0$ for all $\lambda,\mu\in X$.
\end{lemma}
(Here, $\Ext_\CO^1$ means  the Yoneda-extension group.)
\begin{proof} For all $\lambda,\mu\in X$ we have 
$$
\Hom_{\CO}(\Delta(\lambda),\nabla(\mu))\cong\{v\in  \nabla(\mu)_\lambda\mid e_{\alpha}^{(n)}v=0\text{ for all $\alpha\in R^+$, $n>0$}\}
$$
by the universal property of Verma modules. This is a one-dimensional space if $\lambda=\mu$. If $\lambda\not\le\mu$, then $\nabla(\mu)_\lambda=0$ and hence there are no non-trivial homomorphisms from $\Delta(\lambda)$ to $\nabla(\mu)$. From any non-zero homomorphism $f\colon \Delta(\lambda)\to\nabla(\mu)$ we obtain by dualizing a non-zero homomorphism $f^\ast\colon \Delta(\mu)\to\nabla(\lambda)$. Hence the existence of a non-zero homomorphism implies $\lambda\le\mu$ and $\mu\le\lambda$, hence $\lambda=\mu$. We have proven the first statement. 

Let $0\to \nabla(\mu)\to M\to\Delta(\lambda)\to 0$ be an exact sequence. If $\lambda\not <\mu$, then $\lambda$ is a maximal weight of $M$, so the universal property of $\Delta(\lambda)$ implies that this sequence splits. By dualizing we obtain another exact sequence $0\to\nabla(\lambda)\to M^\ast\to\Delta(\mu)\to 0$, which splits if $\mu\not <\lambda$. But  the conditions $\lambda<\mu$ and $\mu<\lambda$ cannot both hold. Hence the second statement is also true.
\end{proof}
\subsection{The Frobenius twist}
The group scheme $G_k$ is induced from an integral version $G_\DZ$ via base change. Hence there exists a Frobenius endomorphism $G_k\to G_k$ (an endomorphism of $k$-group schemes). This endomorphism induces an endomorphism on its $k$-algebra of distributions. We hence obtain an endomorphism    $\Fr\colon U\to U$ that is given by the following formulas: 
\begin{align*}
e_\alpha^{(n)}\mapsto
\begin{cases}
e_\alpha^{(n/p)},&\text{ if $p$ divides $n$},\\
0,&\text{ if $p$ does not divide $n$},
\end{cases}\\
f_\alpha^{(n)}\mapsto
\begin{cases}
f_\alpha^{(n/p)},&\text{ if $p$ divides $n$},\\
0,&\text{ if $p$ does not divide $n$},
\end{cases}\\
{h_\alpha\choose n}\mapsto
\begin{cases}
{h_\alpha\choose n/p},&\text{ if $p$ divides $n$},\\
0,&\text{ if $p$ does not divide $n$}.
\end{cases}
\end{align*}

If $M$ is a $U$-module, then we denote by $M^{[1]}$ the $U$-module that we obtain from $M$ by pulling back the action homomorphism along $\Fr$. If $M$ is a weight module, then so is $M^{[1]}$ and we have $M^{[1]}_\lambda=0$ unless $\lambda\in pX$, in which case $M^{[1]}_\lambda=M_{\lambda/p}$. If $M$ is $U^+$-locally finite, then so is $M^{[1]}$ and we obtain a functor $(\cdot)^{[1]}\colon \CO\to\CO$. This is called the {\em Frobenius twist}. We denote by $(\cdot)^{[l]}$ the functor obtained by applying $(\cdot)^{[1]}$ $l$-times. As $\Fr$ is surjective, $L^{[1]}$ is simple for all simple objects $L$ in $\CO$. A quick check on the highest weights yields the following result.
\begin{lemma}\label{lemma-FrobTwistSimples} 
We have $L(\lambda)^{[1]}=L(p\lambda)$ for all $\lambda\in X$. 
\end{lemma}
Hence  $L(p\lambda)_\mu=0$ unless $\mu\in pX$. 
\subsection{Direct decompositions}
The goal of this section is to show that we have a Krull-Remak-Schmidt decomposition for all finitely generated objects in $\CO$. Once the analogue of the Fitting lemma is proven, the arguments for this result are standard. 

\begin{lemma}\label{lemma-Fitting} Suppose that $M$ is an object in $\CO$ that is finitely generated as a $U$-module. Then any endomorphism $f$ of $M$ induces a Fitting-decomposition, i.e. there exists some $n>0$ such that $M=\im f^n\oplus \ker f^n$.
\end{lemma}
\begin{proof} If $M$ is finitely generated, then every weight space of $M$ is finite dimensional and there exists a finite subset $B$ of $X$ such that $M_B:=\bigoplus_{\lambda\in X} M_\lambda$ is finite dimensional and generates $M$. We denote by $f_B$  the vector space endomorphism on $M_B$ that is induced by $f$. Then there is some $n\gg 0$ such that $\im f^n_B=\im f^{n+1}_B=\im f^{n+2}_B=\dots$. As $M_B$ generates $M$, $\im f_B^l$ generates $\im f^l\subset M$ for all $l>0$ and  we deduce that the image of $f$ stabilizes:  $\im f^n=\im f^{n+1}=\im f^{n+2}=\dots$. Now let $\mu\in X$ and denote by $f_\mu$ the endomorphism on $M_\mu$ induced by $f$. Then  $\im f^n_\mu=\im f^{n+1}_\mu=\dots$. As $M_\mu$ is finite dimensional we also deduce $\ker f^n_\mu=\ker f^{n+1}_\mu=\dots$. We also obtain the Fitting-decomposition $M_\mu=\ker f_\mu^n\oplus\im f_\mu^n$ on each weight space. Hence  $M=\im f^n\oplus \ker f^n$.
 \end{proof}
 
 \begin{lemma}\label{lemma-endlocal} Suppose that $M$ is a finitely generated, indecomposable object in $\CO$. Then every endomorphism of $M$ is either nilpotent or an automorphisms. In particular, $\End_\CO(M)$ is a local ring.
 \end{lemma}
 \begin{proof} Let $f$ be an endomorphism of $M$ that is not an automorphism. The finite dimensionality of the weight spaces allows us to deduce that $f$ is not injective. Hence $\ker f\ne \{0\}$ and we deduce $M=\ker f^n$ for some $n\gg 0$ from the indecomposability of $M$ and Lemma \ref{lemma-Fitting}. Hence $f$ is nilpotent. So every endomorphism is either nilpotent or an automorphism. 
 
 If $f$ is nilpotent and $g$ is an automorphism, then $g-f$ is invertible, hence an automorphism. It follows that the sum of two nilpotent endomorphisms of $M$ is again nilpotent. Moreover, $g\circ f$ and $f\circ g$ are not injective, hence nilpotent. So the nilpotent endomorphisms form an ideal in $\End_\CO(M)$ and each element in its complement is invertible. Hence $\End_\CO(M)$ is a local ring.   \end{proof}
  
\begin{proposition}\label{prop-KRS} Let $M$ be a finitely generated object in $\CO$. Then $M$ can be written as a finite direct sum of indecomposable objects in $\CO$. Moreover, such a decomposition is unique up to reordering and isomorphisms.
\end{proposition}
\begin{proof} As in the proof of Lemma \ref{lemma-Fitting} we denote by $B$ a subset of  $X$ such that $M_B=\bigoplus_{\lambda\in B}M_\lambda \subset M$ is finite dimensional and generates $M$. Then every direct summand of $M$ must intersect $M_B$ non-trivially. In particular, the number of direct summands in a decomposition of $M$ is bounded by the dimension of $M_B$. So a decomposition of $M$ into indecomposable objects  exists.  The uniqueness of this decomposition can be proven (with Lemma \ref{lemma-endlocal}) using standard arguments.
\end{proof}
\section{Verma flags and projective objects in $\CO$}
 The modular category $\CO$ does not possess enough projectives. In fact, only the $L(\lambda)$ with antidominant $\lambda$ admit a projective cover. But locally, i.e. for certain {\em truncated subcategories} $\CO^J$, projective covers exist, as we will show in this section. This  will be good enough for all our purposes. Note that this situation resembles the (related) case of representations of Kac--Moody algebras. 
We will also prove an appropriate version of the BGGH-reciprocity theorem that was first proven in the context of  modular Lie algebras by Humphreys in \cite{H2}. Later, the characteristic $0$ case was treated in \cite{BGG}. Many of the ideas used in the following originate in the article \cite{RCW}. 
\subsection{ Verma flags}
Let $M$ be an object in $\CO$. 
\begin{definition} $M$ is said to {\em admit a Verma flag} if there is a (finite) filtration
$$
0=M_0\subset M_1\subset\dots\subset M_n=M
$$
and $\lambda_1,\dots,\lambda_n\in X$ such that 
 $M_{i}/M_{i-1}\cong\Delta(\lambda_i)$ for all $i=1,\dots,n$. 
\end{definition}
A filtration like the one in the definition above is sometimes, for example in \cite{HBook}, called a {\em standard filtration}.

In the situation of the definition above we set  
$$
(M:\Delta(\mu))=\#\{i\in\{1,\dots,n\}\mid \lambda_i=\mu\}.
$$
This number is independent of the chosen filtration and is called the {\em multiplicity of $\Delta(\mu)$ in $M$}.  
 

 \begin{lemma} \label{lemma-reorderVF} Suppose that $M$ admits a Verma flag. Then there exists a filtration
 $$
 0=M_0\subset M_1\subset\dots\subset M_n=M
 $$
 and $\lambda_{1},\dots,\lambda_n\in X$ such that $M_{i}/M_{i-1}\cong\Delta(\lambda_i)$ for $i=1,\dots,n$, and such that $\lambda_i>\lambda_j$ implies $i<j$.
 \end{lemma}
 \begin{proof} This follows, by induction, from the fact that a surjective homomorphism $N\to\Delta(\lambda)$ splits if $\lambda$ is maximal among the weights of $N$ (by Proposition  \ref{prop-basicVerma}).  \end{proof}

\begin{lemma} \label{lemma-dirsumVF} Let $M=A\oplus B$ be a decomposition in $\CO$. Then $M$ admits a Verma flag if and only if $A$ and $B$ admit Verma flags. In this case,
$$
(M:\Delta(\mu))=(A:\Delta(\mu))+(B:\Delta(\mu)).
$$
\end{lemma}
\begin{proof} If $A$ and $B$ admit Verma flags, then it is easily shown that $M$ admits a Verma flag and that the claim about the multiplicities holds. It is hence enough to show that if $M$ admits a Verma flag, then $A$ and $B$ do as well. This is proven using the same arguments as in the characteristic $0$ case that can be found, for example, in Section 3.7 in \cite{HBook}.
\end{proof}

\subsection{Some functors on $\CO$}
  It is now useful to endow the set $X$ with a topology. 
 \begin{definition}
 \begin{enumerate}
 \item  We say that a subset $ J $ of $X$ is {\em open} if $\lambda\in  J $ and $\mu\le\lambda$ imply $\mu\in J $.
 \item We say that a subset $ I $ of $X$ is {\em closed} if $\lambda\in  I $ and $\lambda\le\mu$  imply $\mu\in I $.
 \item We say that a subset $K$ of $X$ is {\em locally closed} if $\lambda,\nu \in K$ and $\lambda\le\mu\le\nu$ imply $\mu\in K$.
\end{enumerate}
\end{definition}
This indeed defines a topology. Note that arbitrary unions of closed subsets are closed again. 

Let $M$ be an object in $\CO$, and let $ I \subset X$ be a closed subset with open complement $ J $. We set
\begin{align*}
M_I  &:=\left \lgl \bigoplus_{\lambda\in  I } M_\lambda\right\rgl_U\subset M,\\
M^J &:=M/M_I  ,
\end{align*}
i.e.  $M_I  \subset M$ is the minimal submodule of $M$ that contains the weight spaces $M_\lambda$ with $\lambda\in  I $,  and $M^J $ the largest quotient of $M$ with  all weights contained in $ J $. Clearly, these definitions yield functors $(\cdot)_I  ,(\cdot)^J \colon\CO\to\CO$.

\begin{proposition}\label{prop-VermaSub}  Let $M$ be an object in $\CO$. Then the following are equivalent.
\begin{enumerate}
\item $M$ admits a Verma flag.
\item $M^J $ admits a Verma flag for any open set $ J $.
\item $M_I  $ admits a Verma flag for any closed set $ I $.
\end{enumerate}
If either statement above is true, then we have
\begin{align*}
(M^J :\Delta(\nu))&=
\begin{cases} 
(M:\Delta(\nu)),&\text{ if $\nu\in J $},\\
0,&\text{ otherwise,}
\end{cases}\\
(M_I  :\Delta(\nu))&=
\begin{cases} 
(M:\Delta(\nu)),&\text{ if $\nu\in I $},\\
0,&\text{ otherwise}
\end{cases}
\end{align*} 
for all open subset $ J $ and all closed subsets $ I $ of $X$.
\end{proposition}

\begin{proof} As $M=M_X=M^X$, either (2) or (3)  imply (1). 
Suppose $M$ admits a Verma flag. Let $ J $ be an open subset of $X$ with closed complement $ I $. Using Lemma \ref{lemma-reorderVF} we deduce that there is a submodule $M^\prime$ of $M$, appearing in a Verma flag of $M$, such that $M^\prime$ and $M/M^\prime$ admit Verma flags and such that 
$$
(M^\prime:\Delta(\lambda))=\begin{cases}
(M:\Delta(\lambda)),&\text{if $\lambda\in  I $}, \\
 0,& \text{if $\lambda\in J $},
 \end{cases}
 $$
 and
 $$
 (M/M^\prime:\Delta(\lambda))=\begin{cases}
(M:\Delta(\lambda)),&\text{if $\lambda\in  J $}, \\
 0,& \text{if $\lambda\in I $}.
 \end{cases}
 $$
 In particular, it follows that $M^\prime=M_I  $ and $M/M^\prime=M^J $. This shows that (1) implies (2) and (3), and that the statement about the multiplicities is true if (1) holds.  \end{proof}
\subsection{Objects admitting a Verma flag as $U^{\le 0}$-modules} 
Set $U^{\le 0}:=U^-U^0$. Again this is a subalgebra of $U$.  We denote by $U^{\le 0}\catmod^{wt}$ the full subcategory of the category of $U^{\le 0}$-modules that contains all objects that are weight modules, and we denote by $\Res$ the restriction functor from the category $\CO$ to $U^{\le 0}\catmod^{wt}$.

\begin{lemma} \label{lemma-Deltaproj} Let $\lambda\in X$. Then $\Res\Delta(\lambda)$ is a projective object in $U^{\le 0}\catmod^{wt}$.
\end{lemma}
\begin{proof}
Let $f\colon M\to\Res\Delta(\lambda)$ be a surjective homomorphism in $U^{\le 0}\catmod^{wt}$. 
 Let $g_\lambda\colon\Res \Delta(\lambda)_\lambda\to M_\lambda$ be a right inverse  of $f_\lambda\colon M_\lambda\to\Res\Delta(\lambda)_\lambda$ (in the category of vector spaces). Then $g_\lambda$ is a $U^0$-module homomorphism. As $\Delta(\lambda)$ is free as a $U^-$-module of rank 1 (by Kostant's version of the PBW-theorem) and as it is generated by any non-zero element in the one-dimensional space $\Delta(\lambda)_\lambda$, we have $\Hom_{U^{\le 0}}(\Res\Delta(\lambda),M)=\Hom_{U^0}(\Res\Delta(\lambda)_\lambda,M_\lambda)$. So $g_\lambda$ induces a morphism $g\colon \Res\Delta(\lambda)\to M$ with the property that its $\lambda$-component is right inverse to $f_\lambda$. Hence $f\circ g$ is an endomorphism of $\Res\Delta(\lambda)$ that restricts to the identity on $\Res\Delta(\lambda)_\lambda$. Hence $f\circ g$ is the identity, so $f$ splits.  The claim follows.
\end{proof}

\begin{lemma} \label{lemma-VFsplits}
\begin{enumerate}
\item Let $N$ be an object in $\CO$ that admits a Verma flag. Then $\Res N$ is isomorphic to a direct sum of objects of the form $\Res \Delta(\lambda)$ for various $\lambda$. 
\item Let $M$ and $N$ be objects in $\CO$ and assume that $N$ admits a Verma flag and let  $f\colon M\to N$ be a surjective homomorphism. Then $\Res f\colon \Res M\to \Res N$ splits. 
\end{enumerate}
\end{lemma}
\begin{proof} 
Statement (1) follows from the fact that each $\Res \Delta(\lambda)$ is projective in $U^{\le0}\catmod^{wt}$. From (1) and Lemma \ref{lemma-Deltaproj} it then also follows that $\Res N$ is projective in $U^{\le 0}\catmod^{wt}$, which implies statement (2).
\end{proof}

\subsection{Tensor products}

Let $L$ be an  object in $\CO$. Then the functor  $M\mapsto M\otimes L$ preserves the category of weight modules and the category of $U^+$-locally finite modules, hence it induces a functor from $\CO$ to $\CO$.

\begin{lemma}\label{lemma-TransVermaMult} Suppose that $L$ is finite dimensional. If $M$ admits a Verma flag, then so does $M\otimes L$  and we have 
$$
(M\otimes L:\Delta(\mu))=\sum_{\lambda\in X} (M:\Delta(\lambda))\dim_kL_{\mu-\lambda}.
$$
\end{lemma}

\begin{proof} Note that the functor $(\cdot)\otimes L$ is exact. It is hence enough to prove the statement for $M=\Delta(\lambda)$ for some $\lambda\in X$. In this case, the same arguments can be used as in the characteristic $0$ case, which can be found, for example, in Section 3.6 of \cite{HBook}.
\end{proof}

\subsection{Projective objects in truncated categories}

Let $ J $ be an open subset of $X$. 

\begin{definition} We denote by $\CO^J $ the full subcategory of $\CO$ that contains all objects $M$ that have the property that $M_\lambda\ne 0$ implies $\lambda\in J $.
\end{definition}
Clearly, the functor $M\mapsto M^J $ is actually a functor from $\CO$ to $\CO^J $, and we have $M\in\CO^J $ if and only if  $M=M^J $. Note that $L(\lambda)$ is contained in $\CO^J $ if and only if $\Delta(\lambda)$ is contained in $\CO^J $, which is the case if and only if $\lambda\in J $.

\begin{definition} We say that an open subset  $ J $ of $X$ is {\em quasi-bounded} if  for any $\lambda\in X$ the set $\{\mu\in J \mid \lambda\le\mu\}$ is finite.
\end{definition}
 
 Recall that in an abelian category $\CA$ an epimorphism $f\colon A\to B$ is called {\em essential} if the following holds: if $g\colon C\to A$ is a morphism in $\CA$ such that $f\circ g\colon C\to B$ is an epimorphism, then $g$ is an epimorphism. A morphism $f\colon A\to B$ is called a {\em projective cover} of $B$ if $A$ is projective and $f$ is an essential epimorphism. 

\begin{theorem}\label{thm-projcov} Suppose that $ J \subset X$ is open and quasi-bounded. For any $\lambda\in J $ there exists a projective cover $p_\lambda\colon P^{ J }(\lambda)\to L(\lambda)$ of $L(\lambda)$ in $\CO^{ J }$, and the object  $P^{ J }(\lambda)$ admits a Verma flag.
\end{theorem}

\begin{proof} Consider the $U^{\ge 0}$-module $C(\lambda):=U^{\ge 0}\otimes_{U^0} k_\lambda$. Theorem \ref{thm-PBW} implies that $C(\lambda)$ is free of rank $1$ as a $U^{+}$-module. As an $U^0$-module it is a weight module and its weights are contained in $\{\mu\mid\mu\ge \lambda\}$. As before, define $C(\lambda)^J $  as the largest quotient $U^{\ge 0}$-module with all weights contained in  $ J $, i.e. $C(\lambda)^J=C(\lambda)/\bigoplus_{\mu\not\in J}C(\lambda)_\mu$.  As $ J $ is supposed to be quasi-bounded, $C(\lambda)^{ J }$ is a finite dimensional $U^{\ge 0}$-module and a weight module supported in $\{\mu\mid\mu\in J ,\mu\ge\lambda\}$. It hence has a $U^{\ge0}$-module filtration with one-dimensional subquotients isomorphic to $k_\mu$ for various $\mu\in  J $ (with $\mu\ge\lambda$). Moreover, $k_\mu$ occurs with multiplicity $\dim_k C(\lambda)^{ J }_\mu$.

Now consider $Q^{ J }(\lambda):=U\otimes_{U^{\ge0}} C(\lambda)^J $. This is a weight module, and, as the functor $U\otimes_{U^{\ge0}}(\cdot)$ is exact (by Theorem \ref{thm-PBW}), it has a finite filtration with subquotients isomorphic to $\Delta(\mu)$ for various $\mu\in J $ (with $\mu\ge\lambda$). Hence it is an object in $\CO^J $ that admits a Verma flag and the multiplicities are given by
$$
(Q^{ J }(\lambda):\Delta(\mu))=\dim_k C(\lambda)^{ J }_{\mu}.
$$ 
Now let $N$ be an object in $\CO^{ J }$. We then have a natural isomorphisms  (that are functorial in $N$):
\begin{align*}
\Hom_{\CO^{ J }}(Q^{ J }(\lambda),N)&=\Hom_{U}(U\otimes_{U^{\ge0}} C(\lambda)^J ,N)\\
&=\Hom_{U^{\ge 0}}(C(\lambda)^J ,N)\\
&=\Hom_{U^{\ge0}}(C(\lambda),N) \quad\text{ (as $N=N^J $)}\\
&=\Hom_{U^{\ge0}}(U^{\ge 0}\otimes_{U^0} k_\lambda,N)\\
&=\Hom_{U^{0}}(k_\lambda,N)\\
&\cong N_\lambda.
\end{align*}
The functor $N\mapsto N_\lambda$ is exact on $\CO^{ J }$, so the object $Q^{ J }(\lambda)$ is projective in $\CO^{ J }$. 

By the weight considerations above, $\Delta(\lambda)$ appears in a Verma flag of $Q^{ J }(\lambda)$ with multiplicity $1=\dim_k C(\lambda)_\lambda$ and it is minimal among the highest weights of Verma subquotients. So there exists a surjection $Q^{ J }(\lambda)\to\Delta(\lambda)$ in $\CO^{ J }$. As $Q^{ J }(\lambda)$ is finitely generated (even cyclic) as a $U$-module we can apply Proposition \ref{prop-KRS}, so $Q^J(\lambda)$ splits into a finite direct sum of indecomposables. As a homomorphism $M\to\Delta(\lambda)$ is surjective if and only if it is surjective on the $\lambda$-weight space, and as $\Delta(\lambda)_\lambda$ is of dimension $1$,  there must be an indecomposable direct summand $P^{ J }(\lambda)$ of $Q^{ J }(\lambda)$ that admits a surjection onto $\Delta(\lambda)$. Then $P^{ J }(\lambda)$ is projective and admits a Verma flag by Lemma \ref{lemma-dirsumVF}. 

As $L(\lambda)$ is a quotient of $\Delta(\lambda)$, there exists a surjective homomorphism $f\colon P^{ J }(\lambda)\to L(\lambda)$. 
It remains to show that this is an essential homomorphism. So let $g\colon M\to P^{ J }(\lambda)$ be a morphism such that $f\circ g\colon M\to L(\lambda)$ is an epimorphism. The projectivity of $P^{ J }(\lambda)$ allows us to find a morphism $h\colon P^{ J }(\lambda)\to M$ such that the diagram 

\centerline{
\xymatrix{
P^{ J }(\lambda)\ar[r]^h\ar[dr]_f&M\ar[r]^g\ar[d]^{f\circ g}&P^{ J }(\lambda)\ar[dl]^f\\
&L(\lambda)&
}
}
\noindent
commutes. As $f$ is an epimorphism,  the composition $g\circ h$ cannot be nilpotent, hence must be an automorphism of $P^{ J }(\lambda)$ by Lemma \ref{lemma-endlocal}. In particular, $g$ is an epimorphism. So $f$ is essential and indeed a projective cover. 
\end{proof}

\subsection{Verma multiplicities and the BGGH-reciprocity}
We are now going to prove an analogue of the reciprocity result that appeared in \cite{H2} in the context of representations of modular Lie algebras. 
\begin{proposition}\label{prop-Vermamult} Let $M$ be an object in $\CO$ that admits a Verma flag. Then
$$
(M:\Delta(\mu))=\dim\Hom_{\CO}(M,\nabla(\mu))
$$
for all $\mu\in X$.
\end{proposition}
\begin{proof} Suppose that $M^\prime\subset M$ is a submodule such that $M^\prime$ admits a Verma flag and $M/M^\prime\cong\Delta(\nu)$ for some $\nu\in X$. Hence
$$
(M:\Delta(\mu))=(M^\prime:\Delta(\mu))+\delta_{\mu,\nu}
$$
for all $\mu\in X$. 
Since $\Ext^1_\CO(\Delta(\nu),\nabla(\mu))=0$ by Lemma \ref{lemma-ExtVan},  we have an exact sequence
$$
0\to \Hom(\Delta(\nu),\nabla(\mu))\to \Hom(M,\nabla(\mu))\to \Hom(M^\prime,\nabla(\mu))\to 0,
$$
and, again using Lemma \ref{lemma-ExtVan}, we deduce
$$
\dim_k\Hom(M,\nabla(\mu))=\dim_k\Hom(M^\prime,\nabla(\mu))+\delta_{\nu,\mu}.
$$
The statement follows by induction on the length of a Verma flag.
\end{proof}

Now we are ready to state and prove the mentioned reciprocity result. It appears as statement (3) in the following.
 \begin{proposition} Fix an open subset $ J $ of $X$ and $\lambda\in J $.
\begin{enumerate}
\item The only irreducible quotient of $P^J (\lambda)$ is $L(\lambda)$.
\item Suppose that $M$ is a finitely generated object in $\CO^J $. Then
$$
\dim_k\Hom_{\CO^J}(P^{ J }(\lambda),M)=[M:L(\lambda)].
$$
\item Let $\mu\in X$. Then
$$
(P^J (\lambda):\Delta(\mu))=
\begin{cases}
[\nabla(\mu):L(\lambda)],&\text{ if $\mu\in J $},\\
0,&\text{ if $\mu\not\in J $}.
\end{cases}
$$
\end{enumerate}
\end{proposition}

\begin{proof} Let $M\subset P^J (\lambda)$ be a submodule with $M\ne P^{ J }(\lambda)$.  Let $f\colon P^{ J }(\lambda)\to L(\lambda)$ be a projective cover. Then $f(M)=0$. In particular, if $P^{ J }(\lambda)/M$ is irreducible, then $f$ must induce an isomorphism $P^{ J }(\lambda)/M\cong L(\lambda)$, hence (1). The projectivity of $P^{ J }(\lambda)$ and part (1) imply part (2). Using Proposition \ref{prop-Vermamult} we calculate
\begin{align*}
(P^J (\lambda):\Delta(\mu))&=\dim_k\Hom_{\CO^{ J }}(P^J (\lambda),\nabla(\mu)\\
&=[\nabla(\mu):L(\lambda)]\quad\text{ (by (2)).}\\
\end{align*}
Hence (3).
\end{proof}
Note that $[\nabla(\mu):L(\lambda)]=[\Delta(\mu):L(\lambda)]$ as the characters of the Verma and the dual Verma modules agree. 
\subsection{The exactness of $(\cdot)_I  $ and $(\cdot)^J $} 
Let us denote by $\CV$ the full subcategory of $\CO$ that contains all objects that admit a Verma flag. In general, it does not contain the subobjects and quotients of its objects, so it is not an abelian subcategory. However, it inherits  an exact structure from the surrounding category $\CO$.
 We call a short sequence $0\to A\to B\to C\to 0$ in $\CV$ {\em exact} if it is exact when considered as a sequence in $\CO$. 
 
 Let $ J $ be an open subset of $X$ with closed complement $ I $. Proposition \ref{prop-VermaSub} implies that $(\cdot)_I  $ and $(\cdot)^J $ induce  endofunctors on $\CV$ that we denote by the same symbol. Then we have the following.

\begin{proposition}\label{prop-TruncExact}
Let $0\to A\to B\to C\to 0$ be a sequence in $\CV$. Then the following are equivalent.
\begin{enumerate}
\item The sequence $0\to A\to B\to C\to 0$ is exact.
\item For any open subset $ J $ of $X$, the sequence $0\to A^J \to B^J \to C^J \to 0$ is exact.
\item For any closed subset $ I $ of $X$, the sequence $0\to A_I  \to B_I  \to C_I  \to 0$ is exact.
\end{enumerate}
\end{proposition}
\begin{proof} As $X$ is open in $X$,   property (1) is a special case of either one of the properties (2) or (3). It is hence enough to show that (1) implies (2) and (3). Let $M$ be an object in $\CO$. Then $\bigoplus_{\lambda\in  I }M_\lambda$ is stable under the action of the subalgebra $U^{\ge 0}$ (as $ I $ is closed). Hence $M_I  $ is generated by $\bigoplus_{\lambda\in  I }M_\lambda$ not only as a $U$-submodule, but already as a $U^-$-submodule, and, in particular, as a $U^{\le 0}$-submodule. As the algebra $U^{\le 0}$ detects the weight decomposition, we deduce that the functor $(\cdot)_I  $ factors through the restriction functor $\Res\colon\CO\to U^{\le 0}\catmod^{wt}$. But in the category $U^{\le 0}\catmod^{wt}$ the short exact sequence appearing in (1) splits (cf. Lemma \ref{lemma-VFsplits}). As $(\cdot)_I  $ is additive, it produces an exact sequence when applied to a split sequence. Hence (1) implies (3). By the snake lemma, (1) and (3) imply (2).
\end{proof}

\subsection{Subcategories associated to locally closed subsets}\label{subsec-subcat}
Let $K$ be a locally closed subset of $X$.

\begin{definition} We denote by $\CV({K})$ the full subcategory of category $\CV$ that contains all objects $M$ that have the property that $(M:\Delta(\lambda))\ne 0$ implies $\lambda\in K$.
\end{definition}

Now let $\gamma\in X$ be a {\em dominant} weight. Then $L=L(\gamma)$ is finite dimensional. Denote by $\Gamma=\Hom_k(L,k)$ the dual $U$-module (without the $\sigma$-twist!). It is a finite dimensional module contained in $\CO$ and the irreducible $U$-module of lowest weight $-\gamma$. Let $ J ^\prime\subset J $ be open subsets of $X$. Then $\tJ^\prime= J ^\prime+\gamma\subset\tJ= J +\gamma$ are open in $X$ as well. We define the {\em shift functors}
\begin{eqnarray*}
\bU:=((\cdot)\otimes L)_{X\setminus\tJ^\prime}\colon \CO\to\CO,\\
\bD:=((\cdot)\otimes \Gamma)^{ J }\colon \CO\to\CO.
\end{eqnarray*}

Set $K:= J \setminus J ^\prime$  and $\tK:=\tJ\setminus \tJ^\prime$. Then $K$ and $\tK$ are locally closed, and the map  $(\cdot)+\gamma$ yields a bijection $K\xrightarrow{\sim}\tK$.  

\begin{lemma}\label{lemma-updownVerma}
 Suppose that there exists some $l>0$  such that the triple $( J , J ^\prime,\gamma)$ satisfies the following assumptions:
 \begin{itemize}
 \item  $\gamma\in p^lX$,
 \item  for all $\lambda\in J $ and $\nu>0$ we have $\lambda-p^l\nu\in J ^\prime$, i.e. $K\cap (K-p^l\nu)=\emptyset$. 
\end{itemize} 
Then the following holds.  
\begin{enumerate}
\item Let $\lambda\in K$. Then $\bU(\Delta(\lambda))\cong\Delta(\lambda+\gamma)$.
\item  Let $\tilde\lambda\in\tK$. Then $\bD(\Delta(\tilde\lambda))\cong\Delta(\tilde\lambda-\gamma)$.
\item After restriction, $\bU$ and $\bD$ induce exact functors
$\bU\colon\CV(K)\to\CV(\tK)$ and $\bD\colon\CV(\tK)\to\CV(K)$.
\end{enumerate}
\end{lemma}
\begin{proof} We prove statement (1). Let $\lambda\in K$ and consider the module $\Delta(\lambda)\otimes L$. By Lemma \ref{lemma-TransVermaMult} it admits a Verma flag with multiplicities
$$
(\Delta(\lambda)\otimes L: \Delta(\mu))=\dim_k L_{\mu-\lambda}.
$$
So Proposition \ref{prop-VermaSub} implies that  $\bU(\Delta(\lambda))=(\Delta(\lambda)\otimes L)_{X\setminus\tJ^\prime}$ admits a Verma flag with 
$$
(\bU(\Delta(\lambda)): \Delta(\mu))=
\begin{cases}
 \dim_k L_{\mu-\lambda},&\text{ if $\mu\in X\setminus\tJ^\prime$}\\ 
 0,&\text{ else.}
\end{cases}
$$
Now suppose $\mu$ is such that $(\bU(\Delta(\lambda)): \Delta(\mu))\ne 0$. Then $\mu-\lambda$ is a weight of $L$ and $\mu\not\in\tJ^\prime$.
As $\gamma=p^l\gamma^\prime$ for some $\gamma^\prime\in X$ we have $L=L(\gamma^\prime)^{[l]}$ by Lemma \ref{lemma-FrobTwistSimples} and hence there exists a weight $\eta^\prime$ of $L(\gamma^\prime)$ with $\mu-\lambda=p^l\eta^\prime$. In particular, $\lambda+p^l\eta^\prime\not\in\tJ^\prime$. As $\tJ^\prime= J ^\prime+p^l\gamma^\prime$, this implies that $\lambda-p^l(\gamma^\prime-\eta^\prime)\not\in J ^\prime$. As $\gamma^\prime-\eta^\prime\ge 0$,  our assumption implies that $\eta^\prime=\gamma^\prime$. Hence $\mu=\lambda+p^l\gamma^\prime=\lambda+\gamma$, and, as $\dim_k L_{\gamma}=1$ we obtain statement (1). Statement  (2) is proven with similar arguments. 

Now (3) follows from (1) and (2) and the exactness of the functors $(\cdot)\otimes L$ and $(\cdot)\otimes \Gamma$ together with  Proposition  \ref{prop-TruncExact}. 
\end{proof}

\begin{theorem} \label{thm-equiv1} Suppose that there exists some $l>0$ such  that the triple $( J , J ^\prime,\gamma)$ satisfies the following.
 \begin{itemize}
 \item  $\gamma\in p^lX$,
 \item  for all $\lambda\in J $ and $\nu>0$ we have $\lambda-p^l\nu\in J ^\prime$, i.e. $K\cap (K-p^l\nu)=\emptyset$. 
\end{itemize} 
Then the functors $\bU$ and $\bD$ induce mutually inverse equivalences
$$
\CV(K)\cong\CV(\tK)
$$
of exact categories.
\end{theorem}

\begin{proof} We first construct a transformation $\tau\colon\bD\circ\bU\to\id_{\CV(K)}$ and show that this is an isomorphism, and then we construct a transformation $\sigma\colon\id_{\CV(\tK)}\to\bU\circ\bD$ and show that it is an isomorphism. So first let $M$ be an object in $\CV(K)$. We claim that there is a commutative diagram  of the following form:

\centerline{
\xymatrix{
(M\otimes L)_{X\setminus\tJ^\prime}\otimes \Gamma\ar[d]_{i\otimes \id_{\Gamma}}\ar[r]^b&((M\otimes L)_{X\setminus\tJ^\prime}\otimes \Gamma)^J \ar[d]^{\tau^M}\\
M\otimes L\otimes \Gamma\ar[r]^{\id_M\otimes\ev}&M.
}}\noindent
Here, $\ev\colon L\otimes \Gamma\to k_{triv}$ is the evaluation homomorphism, and $i$ is  the inclusion $(M\otimes L)_{X\setminus\tJ^\prime}\subset M\otimes L$. The composition of these two homomorphisms yields a homomorphism $(M\otimes L)_{X\setminus\tJ^\prime}\otimes L\to M$, and this must factor over the canonical quotient $b\colon (M\otimes L)_{X\setminus\tJ^\prime}\otimes L\to ((M\otimes L)_{X\setminus\tJ^\prime}\otimes L)^J $, as the weights of $M$ are contained in $ J $. The resulting homomorphism  $\tau^M$ is functorial in $M$, and hence we obtain a natural transformation $\tau\colon \bD\circ\bU\to\id_{\CV(K)}$ of endofunctors on $\CV(K)$. 

We now show that this is an isomorphism of functors. As both functors are exact it is sufficient to show that $\tau$ is an isomorphism when evaluated at a Verma module $\Delta(\mu)$ with $\mu\in K$. By Lemma \ref{lemma-updownVerma} we have  $(\bD\circ\bU) \Delta(\mu)\cong\Delta(\mu)$, and hence it suffices to show that the homomorphism $\tau^{\Delta(\mu)}$ is non-zero. For this, we consider again the diagram above, with $M=\Delta(\mu)$. It suffices to show that the composition $(\id_M\otimes \ev)\circ (i\otimes\id_{\Gamma})$ is non-zero. If we denote by $v\in L$ a non-zero vector of (maximal) weight $\gamma$ and $m\in\Delta(\mu)$ a non-zero vector of weight $\mu$, then $m\otimes v$ is contained in $(\Delta(\mu)\otimes L)_{X\setminus\tJ^\prime}$, as it is of weight $\mu+\gamma$ which is not contained in $\tJ^\prime$. For an element $\phi\in \Gamma$ with $\phi(v)\ne 0$ we then have $(\id_{\Delta(\mu)}\otimes\ev)\circ (i\otimes\id_{\Gamma})(m\otimes v\otimes\phi)=\phi(v)m$, which is non-zero. So $\tau^{\Delta(\mu)}$ is indeed non-zero. 

Now let $N$ be an object in $\CV(\tK)$. We claim that there is a commutative diagram  of the following form:

\centerline{
\xymatrix{
N\ar[r]^{\id_N\otimes\coev}\ar[d]_{\sigma^N}&N\otimes \Gamma\otimes L \ar[d]^{b\otimes\id_L}\\
((N\otimes \Gamma)^J \otimes L)_{X\setminus\tJ^\prime}\ar[r]^i&(N\otimes \Gamma)^J \otimes L.
}
}
\noindent
Here, $\coev\colon k_{triv}\to \Gamma\otimes L$ is the coevaluation homomorphism, and $b$ is the canonical quotient $N\otimes \Gamma\to (N\otimes \Gamma)^J $. The resulting homomorphism $N\to (N\otimes \Gamma)^J \otimes L$ factors over the inclusion $i\colon  ((N\otimes \Gamma)^J \otimes L)_{X\setminus\tJ^\prime}\to (N\otimes \Gamma)^J \otimes L$, as the highest weights of the Verma subquotients of $N$ (and hence the weights of some generators of $N$) are supposed to be contained in $\tK$, so are not contained in $\tJ^\prime$. The resulting homomorphism $\sigma^N\colon N\to ((N\otimes \Gamma)^J \otimes L)_{X\setminus\tJ^\prime}=\bU\circ\bD(N)$ is functorial in $N$ and yields a transformation $\sigma\colon \id_{\CV(\tK)}\to\bU\circ\bD$. As both functors are exact, and as $\bU\circ\bD(\Delta(\tilde\lambda))\cong\Delta(\tilde\lambda)$, it suffices, as before, to show that $\sigma^{\Delta(\tilde\lambda)}$ is non-zero for all $\tilde\lambda\in\tK$. 

Consider again the diagram above. It suffices to show that the composition $\Delta(\tilde\lambda)\xrightarrow{\id_{\Delta(\tilde\lambda)}\otimes\coev} \Delta(\tilde\lambda)\otimes \Gamma\otimes L\xrightarrow{b\otimes\id_L} (\Delta(\tilde\lambda)\otimes \Gamma)^J \otimes L$ is non-zero. Fix a basis of weight vectors $\{v_0,\dots,v_d\}$ of $L$ and denote by  $\nu_i$  the weight of $v_i$. Suppose that $\nu_0=\gamma$, so $\nu_i\ne \gamma$ for all $i\ne 0$. Let $\{v_0^\ast,\dots,v_d^\ast\}$ be the dual basis of $\Gamma$. Then $(b\otimes\id_L)\circ (\id_{\Delta(\tilde\lambda)}\otimes\coev)(m)$ is the image of $\sum m\otimes v_i^\ast\otimes v_i$ in $(\Delta(\tilde\lambda)\otimes \Gamma)^J \otimes L$. Suppose that $m\in\Delta(\tilde\lambda)_{\tilde\lambda}$ is non-zero. Then $m\otimes v_i^\ast$ is of weight $\tilde\lambda-\nu_i$. Suppose that the weight $\tilde\lambda-\nu_i$ is contained in $ J $. Then $(\tilde\lambda-\gamma)+\gamma-\nu_i$ is contained in $ J $. As $\tilde\lambda-\gamma\in J $  and as $\gamma=p^l\gamma^\prime$ and $\nu_i=p^l\nu_i^\prime$ our assumption implies that $\nu_i=\gamma$, i.e. $i=0$. Hence $(b\otimes\id_L)\circ (\id_{\Delta(\tilde\lambda)}\otimes\coev)(m)$ is the image of $m\otimes v_0^\ast\otimes v_0$ in $(\Delta(\tilde\lambda)\otimes \Gamma)^J \otimes L$. As in the proof of Lemma \ref{lemma-updownVerma} we see that $(m\otimes v_0^\ast)^J $ is a generator of $(\Delta(\tilde\lambda)\otimes \Gamma)^J \cong\Delta(\tilde\lambda-\gamma)$, hence non-zero. So $(b\otimes\id_L)\circ (\id_{\Delta(\tilde\lambda)}\otimes\coev)(m)$ is non-zero. This finishes the proof.
\end{proof}

\section{Subquotient categories}

Let us recall the notion of subquotient categories, as developed in \cite{G}. Suppose $\CA$ is an abelian category. A full subcategory $\CN$ that has the property that for any short exact sequence $0\to A\to B\to C\to 0$ in $\CA$ we have $B\in\CN$ if and only if $A,C\in\CN$, is called a {\em Serre subcategory}. If $\CN$ is a Serre subcategory of $\CA$, then we construct the Serre quotient $\CA/\CN$ as follows. The class of objects of $\CA/\CN$ is the class of objects in $\CA$, and  the set of homomorphisms is constructed as follows. Suppose that $M_2\subset M_1\subset M$ are two subobjects of $M$ such that $M/M_1$ and $M/M_2$ are contained in $\CN$, and that $N_1\subset N_2\subset N$ are two subobjects of $N$ that are both contained in $\CN$. We then have canonical homomorphisms
$$
\Hom_{\CA}(M,N)\to \Hom_\CA(M_1,N/N_1)\to\Hom_\CA(M_2,N/N_2).
$$
So we can define
$$
\Hom_{\CA/\CN}(M,N)=\varinjlim \Hom_\CA(M^\prime, N/N^\prime),
$$
where $M^\prime$ ranges over all subobjects of $M$ such that $M/M^\prime$ is contained in $\CN$, and $N^\prime$ ranges over all subobjects on $N$ that are contained in $\CN$. This construction comes with a canonical map $\Hom_\CA(M,N)\to\Hom_{\CA/\CN}(M,N)$. There is a  composition law for morphisms in $\CA/\CN$ such that we obtain a {\em 	quotient functor}  $\bT\colon\CA\to\CA/\CN$.
The following statement summarises some of the main results in \cite{G}. 

\begin{theorem}\cite{G} \begin{enumerate} \item The category $\CA/\CN$ is abelian and the functor $\bT\colon\CA\to\CA/\CN$ is exact.
\item We have $\bT(N)=0$ in $\CA/\CN$ if and only if $N\in\CN$. 
\item For any short exact sequence $S$ in $\CA/\CN$ there is a short exact sequence $S^\prime$ in $\CA$ such that $\bT(S^\prime)$ is isomorphic to $S$.
\end{enumerate}
\end{theorem} 


\subsection{Subquotients for the category $\CO$} Let $ J ^\prime\subset J \subset X$ be  open subsets. Then $\CO^{ J ^\prime}$ is a Serre subcategory of $\CO^J $, and we are interested in the quotient $\CO^{ J }/\CO^{ J ^\prime}$. Here we are in a particularly nice situation: let $M$ and $N$ be objects of $\CO^J $. Then there is a {\em minimal} submodule  $M_-$ of $M$ with the property that $M/M_-$ is contained in $\CO^{ J ^\prime}$. It is the submodule $M_{X\setminus J ^\prime}$ defined earlier, i.e. the submodule of $M$ generated by all weight spaces $M_\lambda$ with $\lambda\not\in J ^\prime$. There is also a {\em maximal} subobject $N_+$ of $N$ that is contained in $\CO^{ J ^\prime}$, it is simply the union of all submodules of $N$ contained in $\CO^{ J ^\prime}$. So the limit in the definition of the $\Hom$-spaces in $\CO^{ J }/\CO^{ J ^\prime}$ stabilizes and we obtain
$$
\Hom_{\CO^J /\CO^{ J ^\prime}}(\bT M,\bT N)=\Hom_{\CO^J }(M_-,N/N_+).\leqno{(\ast)}
$$

\begin{proposition} \label{prop-projinquot}  Suppose that $ J $ is quasi-bounded.   Let $\lambda$ be an element in $ J \setminus J ^\prime$. 
\begin{enumerate}
\item The image of the projective cover $P^{ J }(\lambda)\to L(\lambda)$ of $L(\lambda)$ under the functor $\bT $ is a projective cover of $\bT  L(\lambda)$ in $\CO^J /\CO^{ J ^\prime}$. 

\item Let $N$ be an object in $\CO^{ J }$. Then the induced homomorphism 
$$
\Hom_{\CO^{ J }}(P^{ J }(\lambda),N)\to\Hom_{\CO^J /\CO^{ J ^\prime}}(\bT  P^{ J }(\lambda),\bT  N)
$$
is an isomorphism.
\item The category $\CO^J /\CO^{ J ^\prime}$ has enough projectives. 
\end{enumerate}
\end{proposition}

\begin{proof} First we prove statement (2). We adopt the notation from the paragraph before the statement of this proposition. As $\lambda$ is not contained in $ J ^\prime$ and $N_+$ is an object in $\CO^{ J ^\prime}$ we have $(N_+:L(\lambda))=0$, hence there is no non-trivial homomorphism from $P^J (\lambda)$ to $N_+$. From this and the projectivity of  $P^J (\lambda)$  in $\CO^{ J }$ we deduce that the canonical homomorphism $\Hom_{\CO^J }(P^J (\lambda), N)\to \Hom_{\CO^J }(P^J (\lambda), N/N_+)$ is an isomorphism. Now any submodule $A$ of $P^{ J }(\lambda)$ with the property that $P^{ J }(\lambda)/A$ is contained in  $\CO^{ J ^\prime}$ must contain the $\lambda$-weight space of $P^J (\lambda)$. As $P^J (\lambda)$ is generated by this weight space, we have $P^J (\lambda)_-= P^J (\lambda)$. Equation $(\ast)$ now yields (2).

As all short exact sequences in $\CO^{ J }/\CO^{ J ^\prime}$ can be obtained from short exact sequences in $\CO^{ J }$ by applying the quotient functor $\bT$, we deduce that $\bT P^J (\lambda)$ is projective in $\CO^{ J }/\CO^{ J ^\prime}$ from (2). As $\bT$ is exact, the morphism $\bT P^J (\lambda)\to \bT L(\lambda)$ is an epimorphism. Now $\End_{\CO^J /\CO^{ J ^\prime}}(\bT P^J (\lambda))\cong \End_{\CO^J }(P^J (\lambda))$ by (2), so this is a local ring, so $\bT P^J (\lambda)$ is indecomposable. The fact that $\bT P^J (\lambda)\to \bT L(\lambda)$ is a projective cover now can be proven using the arguments in the last paragraph of the proof of Theorem \ref{thm-projcov}.

Statement (3) follows from (1), since  the objects $\{\bT L(\lambda)\}_{\lambda\in J \setminus J ^\prime}$ represent the simple isomorphism classes in $\CO^J /\CO^{ J ^\prime}$.\end{proof}

\subsection{Subquotients and locally closed subsets}
Again we suppose that $ J $ is quasi-bounded. Set $K:= J \setminus J ^\prime$ and consider the composition of functors
$$
\bV\colon \CV(K)\subset\CO^{ J }\xrightarrow{\bT}\CO^{ J }/\CO^{ J ^\prime}.
$$
Note that the projectives $P^J(\nu)$ with $\nu\in K$ are objects in $\CV(K)$, and that 
$$
\Hom_{\CV(K)}(P^{ J }(\lambda),P^{ J }(\mu))=\Hom_{\CO^{ J }/\CO^{ J ^\prime}}(\bT P^{ J }(\lambda),\bT P^{ J }(\mu))
$$
for all $\lambda,\mu\in K$ by Proposition \ref{prop-projinquot}. As the objects $\{\bT L(\lambda)\}_{\lambda\in K}$ represent the simple isomorphism classes in $\CO^{ J }/\CO^{ J ^\prime}$, and as the morphisms $\bT( P^{ J }(\lambda)\to L(\lambda))$ are projective covers, we deduce that $\CO^{ J }/\CO^{ J ^\prime}$ is the abelianization of the exact category $\CV(K)$. In particular, this quotient category does not depend on the choice of $ J $ and $ J ^\prime$, but only on the locally closed subset $K$. So henceforth we write $\CO_{[K]}$ for this category.

\subsection{The shift functor on subquotients} We return to the setup in Section \ref{subsec-subcat}: we fix open subsets $ J ^\prime\subset J $ of $X$ and suppose that $ J$ is quasi-bounded, we choose a dominant weight  $\gamma\in X^+$  and set $\tJ:= J +\gamma$ and $\tJ^\prime:= J ^\prime+\gamma$, $K:= J \setminus J ^\prime$ and $\tK=\tJ\setminus\tJ^\prime=K+\gamma$. We again consider the functor
$$
(\cdot)\otimes L(\gamma)\colon \CO\to\CO.
$$
It maps the category $\CO^{ J }$, resp. $\CO^{ J ^\prime}$, into the category $\CO^{\tJ}$, resp. $\CO^{\tJ^\prime}$  and hence induces a functor
$$
\bZ\colon \CO_{[K]}\to\CO_{[\tK]}.
$$

\begin{theorem} Suppose that there exists some $l>0$ such  that the triple $( J , J ^\prime,\gamma)$ satisfies the following.
 \begin{itemize}
 \item  $\gamma\in p^lX$,
 \item  for all  $\nu>0$ we have $K\cap (K-p^l\nu)=\emptyset$. 
\end{itemize} 
Then the functor $\bZ$ is an equivalence of categories. 
\end{theorem} 
\begin{proof} Consider the following composition of functors:
\begin{align*}
\bV&\colon\CV(K)\to\CO^{ J }\to\CO^{ J }/\CO^{ J ^\prime}=\CO_{[K]},\\
\tbV&\colon\CV(\tK)\to\CO^{\tJ}\to\CO^{\tJ}/\CO^{\tJ^\prime}=\CO_{[\tK]}.
\end{align*}

Let $\tI^\prime$ be the complement of $\tJ^\prime$ in $X$. For all objects $M$ in $\CO^{\tJ}$  the inclusion $M_{\tI^\prime}\subset M$ has a quotient in $\CO^{\tJ^\prime}$, so this inclusion is an isomorphism in $\CO^{\tJ}/\CO^{\tJ^\prime}$. Hence the square on the right of the diagram 

\centerline{
\xymatrix{
\CV(\tK)\ar[r]^{\tbV}&\CO^{\tJ}/\CO^{\tJ^\prime}\ar[r]^{\id}&\CO^{\tJ}/\CO^{\tJ^\prime}\\\
\CV(K)\ar[u]^{\bU=((\cdot)\otimes L(\gamma))_{\tI^\prime}}\ar[r]^{\bV}&\CO^{ J }/\CO^{ J ^\prime}\ar[u]_{((\cdot)\otimes L(\gamma))_{\tI^\prime}}\ar[r]^{\id}&\CO^{ J }/\CO^{ J ^\prime}\ar[u]_{\bZ=(\cdot)\otimes L(\gamma)}
}
}
\noindent
commutes, making the whole diagram commutative.
Now both categories $\CO^{\tJ}/\CO^{\tJ^\prime}$ and $\CO^{ J }/\CO^{ J ^\prime}$ are abelian and have enough projectives. In order to show that $\bZ$ is an equivalence, it suffices to show that is essentially surjective on projectives and fully faithful on projectives. Using Proposition \ref{prop-projinquot} and the above commutative diagram, it suffices to show this statement for the functor $\bU\colon\CV(K)\to\CV(\tK)$. But $\bU$ is an equivalence by Theorem \ref{thm-equiv1}.
\end{proof}


\begin{thebibliography}{GKM98}
\bibitem[A]{A} Andersen, H.H., {\em BGG categories in prime characteristics}, preprint, {\tt arxiv:2106.00057}.
\bibitem[BGG]{BGG} Bernshtein, J.; Gel'fand, I. M.; Gel'fand, S. I., {\em 
A certain category of $\fg$-modules.} (Russian)
Funkcional. Anal. i Prilozen. 10 (1976), no. 2, 1–8. 
\bibitem[G]{G} Gabriel, P., {\em Des cat\'egories ab\'eliennes}, Bull. Soc. Math. France {\bf 90} (1962), 323--448.

\bibitem[H1]{H2} Humphreys, J. E., {\em Modular Representations of Classical Lie Algebras and Semisimple Groups}, Journal of Algebra {\bf 19} (1971), 51--79.
\bibitem[H2]{HBook}  \bysame, {\em Representations of semisimple Lie algebras in the BGG category $\mathcal O$}, 
Graduate Studies in Mathematics, 94. American Mathematical Society, Providence, RI, 2008. xvi+289 pp.

\bibitem[J] {J}  Jantzen, J.C.,  {\em Representations of algebraic groups}, Math. Surveys and Monogr. { \bf 107},  {Second edition},  {Amer.  Math.  Soc., Providence, RI}, {2003}. 

\bibitem[K]{K} Kostant, B., {\em Groups over $\DZ$}, Proc. Sympos. Pure Math., Vol. 9, Amer. Math. Soc., Providence, R.I., 1966, 90--98.

\bibitem[RCW]{RCW}
Rocha-Caridi, A.; Wallach, N.R., {\em Projective modules over graded Lie algebras}, Mathematische Zeitschrift {\bf 180} (1982), 151--177.

\end{thebibliography}
\end{document}